\documentclass{amsart}
\RequirePackage{float}
\RequirePackage{graphicx}
\RequirePackage{amsmath, amsthm, amssymb, amsfonts}
\RequirePackage[colorlinks=true]{hyperref}
\hypersetup{urlcolor=blue, citecolor=blue}
\RequirePackage[nameinlink]{cleveref}
\RequirePackage{indentfirst}
\RequirePackage{color}

\newcounter{part0}[subsection]
\renewcommand{\part}[1][]{\noindent\stepcounter{part0}{\bfseries Part \number\value{part0}:} #1.\par}
\newcounter{part1}[part0]

\newcounter{part2}[part1]

\numberwithin{equation}{section}
\newtheorem{thm}{Theorem}[section]
\newtheorem{hyp}{Hypothesis}
\newtheorem{prop}[thm]{Proposition}
\newtheorem{coro}[thm]{Corollary}
\newtheorem{lemma}[thm]{Lemma}
\newtheorem{clm}[thm]{Claim}
\theoremstyle{remark}
\newtheorem{remark}{Remark}[section] 
\theoremstyle{definition}
\crefname{equation}{}{}
\crefname{subpart}{Part}{Part}
\newenvironment{Eq}{\begin{equation}\begin{aligned}}{\end{aligned}\end{equation}\ignorespacesafterend}
\newenvironment{Eq*}{\begin{equation*}\begin{aligned}}{\end{aligned}\end{equation*}\ignorespacesafterend}
\newcommand{\Th}[1]{Theorem \ref{#1}}
\newcommand{\Le}[1]{Lemma \ref{#1}}
\newcommand{\Se}[1]{Section \ref{#1}}
\newcommand{\Cl}[1]{Claim \ref{#1}}
\newcommand{\Co}[1]{Corollary \ref{#1}}
\renewcommand{\Pr}[1]{Proposition \ref{#1}}

\DeclareMathOperator{\supp}{supp}
\renewcommand{\d}{\, \mathop{\!\mathrm{d}}\!}
\newcommand{\dyw}{:=}
\newcommand{\DS}{W\!S}
\newcommand{\R}{{\mathbb{R}}}
\newcommand{\N}{{\mathbb{N}}}
\newcommand{\Z}{{\mathbb{Z}}}
\renewcommand{\H}{{\dot H}}
\renewcommand{\l}{{\dot l}}
\newcommand{\B}{{\dot B}}
\renewcommand{\L}{{\fl{L}}}
\newcommand{\Roma}[1]{\uppercase\expandafter{\romannumeral#1}}
\newcommand{\fl}[1]{\mathcal{#1}}
\renewcommand{\k}[3]{\mathopen{}\left#1 #2 \right#3}
\newcommand{\g}{{\mathbf{g}}}
\newcommand{\ga}{{\boldsymbol{\gamma}}}
\newcommand{\m}{{\mathbf{m}}}
\renewcommand{\b}{{\mathbf{b}}}
\newcommand{\ppsi}{\bar\psi}
\newcommand{\hiota}{{\hat\iota}}

\newcommand{\al}{\alpha}    \newcommand{\be}{\beta}
\newcommand{\de}{\delta}    
  \newcommand{\ep}{\varepsilon}

\newcommand{\les}{{\lesssim}}
\newcommand{\pa}{\partial}

\title[Strauss type wave system on asymptotically flat space-time]
{Lifespan of solutions to the Strauss type wave system on asymptotically flat space-time}

\author{Wei Dai}
\address{School of Mathematical Sciences\\ Zhejiang University\\ Hangzhou 310027,P.R.China}
\email{daiw16@zju.edu.cn}

\author{Daoyuan Fang}
\address{School of Mathematical Sciences\\ Zhejiang University\\ Hangzhou 310027,P.R.China}
\email{dyf@zju.edu.cn}

\author{Chengbo Wang}
\address{School of Mathematical Sciences\\ Zhejiang University\\ Hangzhou 310027,P.R.China}
\email{wangcbo@zju.edu.cn}
\urladdr{http://www.math.zju.edu.cn/wang}

\date{\today}

\begin{document}

\begin{abstract}
By assuming certain 
local energy estimates on $(1+3)$-dimensional asymptotically flat space-time, we study the existence portion of the  \emph{Strauss} type wave system.
Firstly we give a kind of space-time estimates which are related to the local energy norm that appeared in \cite{MR2944027}.
These estimates can be used to prove a series of weighted \emph{Strichartz} and \emph{KSS} type estimates, for wave equations on asymptotically flat  space-time.
Then we apply the space-time estimates to obtain the lower bound of the lifespan when the nonlinear exponents $p$ and $q\ge 2$. 
In particular,
our bound for the subcritical case 
 is sharp in general and
we extend the  known region of $(p,q)$ to admit global solutions.
In addition, 
 the initial data are not required to be compactly supported, when  $p, q>2$.
\end{abstract}



\maketitle

\section{Introduction}

In this paper,
we are interested in the longtime solvability of small-amplitude solutions for the Cauchy problem to some coupled system of semilinear wave equations, posed on  asymptotically flat manifolds.
Let $(M, \g)$ be a $(1+3)$-dimensional asymptotically flat space-time manifold and $\square_\g=\nabla^\mu\nabla_\mu$ be the associated \emph{d'Alembertian} operator,
in which $M=\R^+\times\R^3$ or  $\R^+\times(\R^3\backslash\fl{K})$ with bounded smooth $\fl{K}\subset B(0, R_0)$ for some $R_0>0$,
we will study the system
$$\square_\g u_1=F_{p_1}(u_2), \ \square_\g u_2=F_{p_2}(u_1),$$
where the power type nonlinearities $F_{p}\in C^2$ are assumed to satisfy
\begin{equation}
  \sum_{0\leq j\leq 2} |u|^j |\partial_u^j F_p(u)|\lesssim |u|^p \ \mathrm{ for }\ |u|\ll 1.
\end{equation}
Typical examples include $F_p(u)=\pm |u|^p$ and $F_p(u)=\pm |u|^{p-1}u$. 
It is clear that such a system  is closely related to the \emph{Strauss} conjecture. 

For simplicity of presentation,
we set $\iota\in\{1,2\}$, $\hiota=3-\iota$,
and rewrite the system as follows
\begin{Eq}\label{E1.2}
\begin{cases}
\square_\g u_\iota= F_{p_\iota}(u_\hiota),~(t,x)\in M,\\
(u_\iota,\partial_t u_\iota)|_{t=0}=(f_\iota,g_\iota),
\end{cases}
\qquad \iota=1,2,
\end{Eq}
with sufficiently nice and small initial data $f_\iota, g_\iota$. 

Let $p_1, p_2\ge 2$,
under some natural hypotheses on the metric $\g$,
we shall show that the solution to this system is global if $(p_1,p_2)$ is above certain critical curve. 
On the other hand, 
when $(p_1,p_2)$ is below or on this curve, 
it is known that the problem does not admit global solutions in general, 
and we shall prove certain lower bound for the lifespan, 
which is expected to be sharp at least in the non-critical case.

\subsection{History and some discussion}\label{S1.1}~

When $(M,\g)$ is the standard \emph{Minkowski} space-time  $(\R^{3+1},\m)$ and $u_1=u_2$, 
the problem can be reduced to the problem $\square u=F_p(u)$, 
which is the topic of the \emph{Strauss} conjecture and it is known that the critical power is $p=1+\sqrt{2}$ (\cite{MR535704}).
The study of such problem in general spatial dimensions has gone through a long history and has been almost done,
we refer the interested readers to \cite{MR3724252},
\cite{MR3691393}, and  references therein for the current state of the art.

As for the coupled system  \cref{E1.2}, 
it has been well-investigated
 for the \emph{Minkowski} space-time $(\R^{3+1},\m)$, in
\cite{MR2033494},
\cite{MR1785116}.
Let
\begin{Eq}\label{E1.3}
\sigma(p_1,p_2)=\max_{\iota}\k({\frac{p_\hiota+2+\frac{1}{p_\iota}}{p_1p_2-1}-1}),
\end{Eq}
and $T_\m(\varepsilon)$ be the lifespan of solution to \cref{E1.2} with $F_p(u)=|u|^p$ and initial data of size $\varepsilon$, it is known  that
\begin{Eq*}
\begin{cases}
T_\m(\varepsilon)\approx\varepsilon^{-\frac{1}{\sigma}}&\sigma>0,\\
\ln(T_\m(\varepsilon)) \approx\varepsilon^{-\min_\iota p_\iota(p_1p_2-1)}&\sigma=0,p_1\neq p_2,\\
\ln(T_\m(\varepsilon)) \approx\varepsilon^{- p_\iota(p_\iota-1)}&\sigma=0,p_1= p_2,\\
T_\m(\varepsilon)=\infty& \sigma<0,
\end{cases}
\end{Eq*}
where the lower bound was proven for 
compactly supported small data,
while the upper bound was obtained for data with certain positive conditions.

When it comes to general space-time,
the problem becomes more difficult. 
One of the main difficulties in this situation is that, 
 the fundamental solution is hard to obtain, and even the Fourier analysis is not easy to carry out.
Meanwhile, 
since the characteristic surface is not simply conical, 
many of the obvious conclusions in flat space-time may not be valid for the general case. 

However, 
when assuming that the metric is asymptotically flat, 
the situation becomes relatively well.
Recently, there have been many interesting advances in this situation, 
but we only discuss here the most relevant results for the sake of convenience. 

An important estimate that has permitted such progress is a class of weighted Strichartz estimates, which was developed independently in
\cite{MR2769870}
and \cite{HMSSZ}. It has been shown to be robust under small, asymptotically flat perturbations, see, e.g.,
\cite{SoWa10},
\cite{LMSTW},
 \cite{MR3724252},
\cite{MR3691393}.

The weighted Strichartz estimates are known to be closely related to the 
local energy estimates.
Among many versions of local energy estimates,  
a sharp version of
uniform energy and 
 (micro-localized) local energy estimates was established for wave operators with small, asymptotically flat metric perturbations, in \cite{MR2944027}.
By exploiting it with the trace estimate and some other technical tools, 
the global existence for the \emph{Strauss} problem  in the supercritical case 
($p>1+\sqrt{2}$) is obtained \cite{MR3724252},
for a large class of  asymptotically flat,  space-time manifolds. 
On the other hand, in \cite{MR3691393},
 the local energy estimates of \cite{MR2944027} were further exploited and combined with \emph{KSS} type estimates and weighted trace estimates. Based on these estimates, 
certain lower bound estimates of the lifespan were obtained in the subcritical and critical cases, for the \emph{Strauss} problem. 

Recently, the similar argument was applied for the coupled system in \cite{MR3743157}, and
it is proven  that the solution is global when $\sigma<0$ and $p_\iota>2$, for any initial data which are sufficiently regular and small. It is to be remarked that the case $p_1=2$ and $p_2>3.5$ was excluded in \cite{MR3743157}, for which it is known to be admissible for global results on
\emph{Minkowski} space-time.

In this paper, we
synthesize the approach that appeared in these works further to include general weighted space-time estimates, with fewer restrictions on the exponents,
which is the main departure of our approach.
With the help of these estimates,
 we could get more existence results. In particular, we prove global existence when $\sigma<0$ and $p_\iota\ge 2$.

\subsection{Hypotheses}~

We list some hypotheses which may be adopted in our paper. 
Here we should mention that these hypotheses are satisfied in many general models, 
see \cite{MR3691393} for a detailed discussion.
\begin{hyp}[Space-time assumption]
We shall assume $\g=\g_{\mu\nu}(t,x)\d x^\mu \d x^\nu$ is asymptotically flat in the following sense.
We first assume that $\g$ can be decomposed as
\begin{equation}\tag{H1} \label{H1}
  \g = \m + \ga^1(t,r)+\ga^2(t,x),
\end{equation}
where $\m$ denotes the \emph{Minkowski} metric, 
$\ga^1$ is a radial long range perturbation, 
and $\ga^2$ is a short range perturbation.  
More specifically, 
we assume
$$
\sum_k	\|
	(1+|x|)^{|\alpha|+j-1}
	\partial^{\alpha}\ga^j_{\mu\nu}\|_{L_{t,x}^\infty(1+|x|\sim 2^k)}\leq 
C_{\alpha,j},\qquad j=1,2.
$$
The long range perturbation is radial in the sense that when writing out the metric $\ga^1$ in polar coordinates $(t,x)=(t,r\omega)$ with $\omega\in S^{n-1}$, 
we have
\begin{Eq*}
\ga^1(t,r)=\ga^1_{tt}(t,r)\d t^2+2\ga^1_{tr}(t,r)\d t\d r+\ga^1_{rr}(t,r)\d r^2+\ga^1_{\omega\omega}(t,r)\d \omega^2.
\end{Eq*}
\end{hyp}

\begin{hyp}[Local energy assumption]
We assume that there exists an $R_1 > R_0$, 
such that for any solution to the linear equation $\square_\g u = F$, 
we have the uniform energy estimates and weak local energy estimates, 
which is
\begin{equation}\tag{H2}  \label{H2}
\|\partial^{\leq k}u\|_{L\!E_\g}\lesssim \|\partial^{\leq 1}u(0,x)\|_{H_x^k}+\|F\|_{L_t^1H_x^k},
\end{equation}
for any $k\geq 0$, with 
\begin{Eq}\label{E1.4}
\|u\|_{L\!E_\g}\dyw \|\partial u\|_{L_t^\infty L_x^2}+\|\ppsi_{R_1} \partial u\|_{l_\infty^{-\frac{1}{2}}L_t^2L_x^2}+\|u\|_{l_\infty^{-\frac{3}{2}}L_t^2L_x^2},
\end{Eq}
where $\|f\|_{l_\infty^s L^2_tL^2_x}=
\sup_{k\ge 0}	\|2^{sk}	f\|_{L_{t,x}^2(1+|x|\sim 2^k)}$, and the smooth function
$\ppsi_{R_1}$ equals $1$
for $r\geq R_1+1$ and is supported in $r\ge R_1$.
\end{hyp}

\begin{hyp}[Stationary and split metric assumption]
We assume that there exists a $R_2 > R_0$, 
such that
\begin{equation}\tag{H3}  \label{H3}
\g=\g_{00}(x)\d t^2+\g_{ij}(x)\d x^i\d x^j,\qquad r> R_2.
\end{equation}
\end{hyp}

\subsection{Main results}~

Before giving our existence results for \cref{E1.2},
we present a kind of weighted Strichartz estimates, which is our first main result, for general spatial dimensions $n\ge 2$.

\begin{thm}\label{T1.1}
Let $n\geq 2$, we set for $T\in (0,\infty)$
\begin{Eq*}
A_{\alpha}(T)&\dyw
\begin{cases}
(T+2)^{-\alpha},&\alpha\in(0,\frac{1}{2}],\\
\ln^{-\frac{1}{2}} (T+2),&\alpha=0,\\
1,&\alpha<0.
\end{cases}
\end{Eq*} 
Then for any
$2\leq o\leq q<\infty$, 
$\alpha\leq 1/q$ and $1/2-1/o\leq s\leq (n-\delta)(1/2-1/q)$ for some $\delta>0$, 
we have
\begin{Eq}\label{E1.5}
\|u\|_{\DS_{q,o,\alpha,s}}\dyw A_{q\alpha/2}(T)^{\frac{2}{q}}\|r^{\frac{n}{2}+\alpha-\frac{1}{q}-\frac{n}{o}-s}\ppsi_1u\|_{L_t^q \L_r^o L_\omega^2}\lesssim\|u\|_{W^s\cap X^s},
\end{Eq}
where $\L_r^o=L^o((0,\infty), r^{n-1}dr)$, $W^s$ and $X^s$ are defined
in \cref{E2.1}, which are basically energy and local energy norm at the regularity level $s$.
Moreover, we have
\begin{Eq}\label{E1.6}
\|r^{\frac{n}{2}-\frac{1}{q}-\frac{n}{o}-s}u\|_{L_t^q \L_r^o L_\omega^2}\lesssim\|u\|_{W^s\cap X^s},
\end{Eq}
with $2\leq o,q\leq \infty$ and ${1}/{2}-{1}/{o}<s<n/2-1/q$. 
\end{thm}

The estimates in \Th{T1.1}, together with \Le{L4.5}, can deduce a series of important weighted \emph{Strichartz} estimates, which, in turn, provide the following desired lower bound of lifespan to problem \cref{E1.2}.

\begin{thm}\label{T1.2}
Let $n=3$, $p_1, p_2\ge 2$, 
\eqref{H1} and \eqref{H2}.
Without loss of generality we assume $p_1\leq p_2$. 
For $\sigma$ defined in \cref{E1.3} we set
\begin{Eq}\label{E1.7}
T_\varepsilon&\dyw
\begin{cases}
c\varepsilon^{-\frac{1}{\sigma}}&\mathrm{if }\ \sigma>0,\\
\exp(c\varepsilon^{-2(p_1-1)})&\mathrm{if }\ \sigma=0,p_1>2,\\
\exp(c\varepsilon^{-2+\delta})&\mathrm{if }\ \sigma=0,p_1=2,\\
\infty &\mathrm{if }~\sigma<0.
\end{cases}
\end{Eq}
Then
the problem \cref{E1.2} admits solution up to $T_\varepsilon$ with any fixed $\delta>0$, 
for 
compactly supported
 initial data, which are sufficiently regular and small of size $\varepsilon$.
Moreover,
if $p_\hiota>2$, 
$(f_\iota,g_\iota)$ do not need to be compactly supported.
In addition, for any $p_1, p_2\ge 2$,
if \eqref{H3} is satisfied, 
 the initial data do not need to be compactly supported.
\end{thm}

\begin{remark}
For more precise statement of \Th{T1.2}, 
see \Th{T4.1}. 
By comparing with the result of \emph{Minkowski} space,  as discussed in \Se{S1.1}, 
we know that the lower bound in our result is sharp in general for the subcritical and supercritical situations. 
But it is still possible to improve it in critical situation.
\end{remark}

\begin{remark}
When $F_{p_1}=F_{p_2}$, \Th{T1.2} yields the result 
for $\square_\g u=F_p(u)$, 
 \Co{C4.2}. For
 $p\in (2, 1+\sqrt{2}]$, this result removes the technical assumption \eqref{H3} for 
 the corresponding result in \cite{MR3691393}.
In particular, we have the improved lower bound of the lifespan
$$T_*\ge \exp(c\varepsilon^{-2\sqrt{2}})$$
for the critical wave equations (with $p=1+\sqrt{2}$), posed on Kerr space-times with small angular momentum
$a\ll M$.
\end{remark}

\begin{remark}
The same results, at least for $p_2\geq p_1>2$, apply for
general operators $P=\square_\m+a^{\mu\nu}\partial_\mu\partial_\nu+b^\mu\partial_\mu+c$, where $a,b,c$ have sufficient decay and regularity.
The only difference in the proof is that the admissible range of $s$ in \Le{L4.5} shrinks to $s\in (-1, 0)$ if there exists zero order term in $P$. 
\end{remark}

\section{Notations}\label{S2}

We list here some notations which will be used.
Firstly, the \emph{Einstein} summation convention is used, 
as well as the convention that Greek indices $\mu, \nu, \cdots$ range from $0$ to $n$ while Latin indices $i, j, \cdots$ will run from $1$ to $n$.

Secondly, the vector fields to be used will be labeled as
\begin{alignat*}{2}
&\partial\dyw \{\partial_\mu\}=(\partial_t, \partial_x ), &\qquad&\Omega\dyw \{\Omega_{jk}=x^k\partial_j-x^j\partial_k\},\\
&Y\dyw\{\partial_x,\Omega\},&&Z\dyw\{\partial, \Omega\}.
\end{alignat*}
For any norm $A$ and a nonnegative integer $k$, 
we shall use the shorthand
\begin{Eq*}
\partial^{\leq k} f \dyw \{\partial^\alpha f\}_{0\leq|\alpha|\leq k},\quad \|\partial^{\leq k} f\|_A= \sum_{|\alpha|\leq k} \|\partial^\alpha f\|_A,
\end{Eq*}
with the obvious modification for other vector fields.

Next, 
we shall use some auxiliary functions. Let $\phi\in C_0^\infty(\R)$ with $\supp\phi\subset[1/2,2]$, 
$\phi_j(x)=\phi(|x|/2^j)$ and $\sum_{j\in \Z}\phi_j^2=1$ for any $0\neq x\in\R^n$. 
Also, we define $S_j$ to be the corresponding homogeneous Littlewood-Paley projection. 
We also fix a class of smooth functions $\psi_R$ satisfying $0\leq\psi_R\leq 1$ and
\begin{Eq*}
\psi_R(x)\dyw
\begin{cases}
1,&r\leq R,\\0,&r\geq R+1,
\end{cases}
\end{Eq*}
meanwhile, 
we denote $\ppsi_R=1-\psi_R$.

Then,
we denote the shorthand for some norm spaces,
\begin{Eq*}
\|f(x)\|_{\L_r^p L_\omega^b}\dyw&
\|\|f(r\omega)\|_{L^b_\omega(S^{n-1})}\|_{L^p_r ((0,\infty), r^{n-1}dr)}
\ , \\
\|f(x)\|_{\l_q^s(A)}\dyw&\|2^{js}\phi_j(x)f(x)\|_{l_{j\in\Z}^qA}\ ,
\end{Eq*}
and $l_q^s(A)$ is defined similarly with the inhomogeneous dyadic decomposition. 
We also denote the $W^s=W_2^s$ norm and $X^s=X_2^s$ norm by
\begin{Eq}\label{E2.1}
&\|f\|_{W_p^s}=\|2^{ks} S_k f\|_{l_{k\in\Z}^p L_t^\infty L_x^2},\qquad \|f\|_{X_p^s}=\|2^{ks} S_k f\|_{l_{k\in\Z}^p X_k},\\
&\|f\|_{X_k}=2^{\frac{k}{2}}\|f\|_{L^2(A_{\leq -k})}+\sup_{j>-k}\||x|^{-1/2} f\|_{L^2(A_{j})}, 
\end{Eq}
with $A_j=\R_+\times\{|x|\sim 2^j\}$, 
$A_{\le -k}=\cup_{j\le -k} A_j,~j, k\in \Z$. 
Moreover, 
we usually omit $T>0$ in the norm when it is clear from the context that 
the norm is taken for $t\in[0,T]$ for the given $T>0$.

Finally, 
for the writing convenience, 
$x\lesssim y$ and $y\gtrsim x$ mean $x\leq Cy$ for some $C>0$, which may change from line to line.
Similarly, $x\approx y$ means that $x\lesssim y\lesssim x$.
We also denote $x/0=\infty$ for any $x>0$ and $x/\infty=0$ for any $-\infty<x<\infty$.  

\section{Proof of \Th{T1.1}}

We begin with a sketch for the proof of \Th{T1.1}. 
First of all, 
we record a property of $X^s$ space. 

\begin{prop}[Lemma 1 of \cite{MR2944027}]
\label{L3.1}
Let $n\geq 2$, 
we have
\begin{Eq}\label{E3.1}
\|u\|_{\l^{-\frac{1}{2}}_\infty L_t^2L_x^2}\lesssim \|u\|_{X^{0}},
\end{Eq}
and when $0<s<{(n-1)}/{2}$ we have
\begin{Eq}\label{E3.2}
\|r^{-\frac{1}{2}-s}u\|_{L_t^2 L_x^2}\lesssim\|u\|_{X^s}.
\end{Eq}
\end{prop}

Then, 
based on the well known \emph{KSS} type estimates, 
for $\alpha\leq 1/2$ we get
\begin{Eq}\label{E3.3}
A_{\alpha}(T)\|\k<r>^{-\frac{1}{2}+\alpha}u\|_{L_t^2L_x^2}\lesssim\|u\|_{L_t^\infty L_x^2\cap l^{-\frac{1}{2}}_\infty L_{t,x}^2}\lesssim\|u\|_{W^0\cap X^0}.
\end{Eq}
Here the last inequality comes from \cref{E3.1} and \emph{Minkowski} inequality, 
and the first inequality follows from the typical \emph{KSS} type estimate, 
see, e.g., 
\cite[Appendix 7.2]{JWY12}, \cite{MR2217314}
 for its proof.

To prove \cref{E1.5}, 
we need the following trace estimates.

\begin{lemma}[Trace estimates]\label{L3.2}
Let $n\geq 2$, 
for $2\leq o\leq \infty$ and $1/2-1/o\leq s<n/2$, 
we have
\begin{Eq}\label{E3.4}
\|r^{\frac{n}{2}-\frac{n}{o}-s}u\|_{\L_r^oL_\omega^2}\lesssim
\begin{cases}
\|u\|_{\dot B_{2,1}^{\frac{1}{2}}}&o=\infty,s=\frac{1}{2},\\
\|u\|_{\H_x^s}&else.
\end{cases}
\end{Eq}
\end{lemma}

Now,
to combine \cref{E3.3} with \cref{E3.4}, 
we need an interpolation property.

\begin{clm}\label{C3.3}
Consider $s_0\neq s_1$, 
set $s=(1-\theta)s_0+\theta s_1$ with $\theta\in(0,1)$, 
for complex interpolation, 
we have
\begin{Eq}\label{E3.5}
W^s\cap X^s= [W^{s_0}\cap X^{s_0},W^{s_1}\cap X^{s_1}]_{\theta},
\end{Eq}
and for real interpolation, 
we have
\begin{Eq}\label{E3.6}
W_p^s\cap X_p^s= (W_{p_0}^{s_0}\cap X_{p_0}^{s_0},W_{p_1}^{s_1}\cap X_{p_1}^{s_1})_{\theta,p},\qquad \forall  p,p_0,p_1\in[1,\infty].
\end{Eq}
Meanwhile,
these results are also correct if we consider the $W^s$ or $X^s$ norm separately.
\end{clm}

Now, based on \cref{E3.4} with $L_t^\infty$ norm on both sides, 
and interpolation with \cref{E3.3},
we get \cref{E1.5}. 
As for \cref{E1.6},  we need a similar trace estimate in $X^s$ space. Firstly we have a weighted trace estimate.

\begin{lemma}\label{L3.4}
Let $n\geq 2$, $o\in[2,\infty]$, 
$\alpha\in(1/2-1/o,1)$ and $\beta\in(\alpha-n/2,n/2)$. 
Then we have
\begin{Eq}\label{E3.7}
\|r^{\frac{n}{2}-\frac{n}{o}-\alpha+\beta}u\|_{\L_r^oL_\omega^2}\lesssim\|r^\beta D^\alpha u\|_{L_x^2}
\end{Eq}
\end{lemma}

Then combining \Le{L3.4} with Proposition \ref{L3.1} we obtain the trace estimate in $X^s$ space.

\begin{lemma}\label{L3.5}
Let $n\geq 2$, 
for $2\leq o\leq \infty$ and $1/2-1/o< s<(n-1)/2$, 
we have
\begin{Eq}\label{E3.8}
\|r^{\frac{n-1}{2}-\frac{n}{o}-s}u\|_{L_t^2\L_r^oL_\omega^2}\lesssim\|u\|_{X^s}.
\end{Eq}
\end{lemma}

Finally, 
interpolating \cref{E3.4} and \cref{E3.8} we get \cref{E1.6} which finishes the sketch. 
In the following subsections, we give the detailed proof for each part. 

\subsection{Proof of \Le{L3.2}}\label{sec-3.1}~


Firstly, 
when $o=\infty$, 
\cref{E3.4} is the well known trace estimate, 
see, 
e.g.,
\cite{MR2769870} and references therein.
When $o=2$, it is the \emph{Hardy}'s inequality and so we have
\begin{Eq}\label{E3.9}
\begin{cases}
\|r^{\frac{n}{2}-s_0}u\|_{\L_r^\infty L_\omega^2}\lesssim\|u\|_{\B_{2,1}^{s_0}},\ s_0\in [1/2, n/2)\ ,\\
\|r^{-s_1}u\|_{\L_r^2 L_\omega^2}\lesssim\|u\|_{\H_x^{s_1}},\  s_1\in [0, n/2).
\end{cases}
\end{Eq}

The estimate for $s=1/2-1/o$ is obtained in
\cite[Theorem 2.10]{MR1386767}, see also
\cite[Proposition 2.2]{MR3552253} for an alternative proof based on interpolation and trace estimate.
 In what follows, we give a unified proof for the case of $o\in (2,\infty)$, in spirit of \cite{MR3552253}.

Let $2< o<\infty$ and $1/2-1/o\leq s<n/2$,
we set
$\theta=2/o\in (0,1)$,
 $s_0=(n - os + nos)/(no - o + 2)$ and $s_1=(n-no/2+nos)/(no - o + 2)$.
 Then a routine calculation shows that
 $s_0>s_1$,
 $s_0\in [1/2, n/2)$, $s_1\in [0, n/2)$,
and \begin{Eq*}
\frac{1-\theta}{\infty}+\frac{\theta}{2}=\frac{1}{o},\qquad (1-\theta)s_0+\theta s_1=s.
\end{Eq*}
With help of these parameters,
 we shall use the real interpolation with parameters $(\theta, o)$ to give the proof,
for which we record the following facts:
\begin{itemize}
\item \cite[Theorem 3.7.1]{MR0482275}: $(A_0,A_1)_{\theta,q}'=(A_0',A_1')_{\theta,q'}$ for $1\leq q<\infty$, $A_0\cap A_1$ is dense in $A_0$ and $A_1$;
\item \cite[Theorem 5.5.1]{MR0482275}: $(L_{(\rho_0\d x)}^{p_0}(A),L_{(\rho_1\d x)}^{p_1}(A))_{\theta,p}=L_{(\rho\d x)}^{p}(A)$ for $1\leq p_0,p_1<\infty$, $1/p=(1-\theta)/p_0+\theta/p_1$ and $\rho=\rho_0^{p(1-\theta)/p_0}\rho_1^{p\theta/p_1}$;
\item \cite[Theorem 6.4.5]{MR0482275}: $(\B_{p,q_0}^{s_0},\B_{p,q_1}^{s_1})_{\theta,q}=\B_{p,q}^{s}$ for $s_0\neq s_1$, $s=(1-\theta)s_0+\theta s_1$ and $1\leq p,q_0,q_1,q\leq \infty$.
\end{itemize}

More specifically,
defining $\|v\|_{A_0}\dyw\|r^{s_0-n/2}v\|_{\L_r^1L_\omega^2}$ and $\|v\|_{A_1}\dyw\|r^{s_1}v\|_{\L_r^2L_\omega^2}$,
by \cref{E3.9} and the first fact,
we conclude
\begin{Eq*}
\|u\|_{(A_0,A_1)_{\theta,o'}'}=\|u\|_{(A_0',A_1')_{\theta,o}}\lesssim\|u\|_{(\B_{2,1}^{s_0},\H_x^{s_1})_{\theta,o}}.
\end{Eq*}
Now,
for the left hand side (LHS),
with $\rho_0\dyw r^{s_0-n/2+n-1}$
and $\rho_1\dyw r^{2s_1+n-1}$,
we set
\begin{Eq*}
\rho\dyw r^{o'(\frac{n}{o}+s-\frac{n}{2})+n-1}=\rho_0^{o'(1-\theta)}\rho_1^{\frac{o'\theta}{2}},
\end{Eq*}
then by the second fact with $p_0=1$,
$p_1=2$,
$p=o'$ and $A=L_\omega^2$,
we have
\begin{Eq*}
\|v\|_{(A_0,A_1)_{\theta,o'}}=\|r^{\frac{n}{o}+s-\frac{n}{2}}v\|_{\L_r^{o'}L_\omega^2}\Rightarrow\|u\|_{(A_0,A_1)_{\theta,o'}'}=\|r^{\frac{n}{2}-\frac{n}{o}-s}u\|_{\L_r^oL_\omega^2}.
\end{Eq*}
For the right hand side (RHS),
noticing that $\H^{s_1}=\B_{2,2}^{s_1}$,
by the third fact we get
\begin{Eq*}
\|u\|_{(\B_{2,1}^{s_0},\H_x^{s_1})_{\theta,o}}=\|u\|_{\B_{2,o}^{s}}\lesssim\|u\|_{\H_x^{s}}
\end{Eq*}
Summing up,
we finish the proof.

\subsection{Proof of \Cl{C3.3}}\label{sec-3.2}~

Similarly to \cite{MR3724252}, 
we define the co-retraction and retraction by
\begin{gather*}
Qf=\k({{\bf1}_{l=k}\phi\k({\frac{\k<{2^kx}>}{2^j}})S_kf})_{j\in\N,k\in\Z,l\in \Z},\\
Q:X_p^s\rightarrow \l_{p(k)}^{s}l_{\infty(j)}^{-\frac{1}{2}}L_{t,x}^2\l_{1(l)}^{\frac{1}{2}}\qquad Q:W_p^s\rightarrow \l_{p(k)}^{s}L_t^\infty L_x^2\l_{1(j)}^{0}\l_{1(l)}^{0},\\
R(a_{jkl})_{j\in\N,k\in\Z,l\in \Z}=\sum_{k\in Z}S_k\sum_{j\geq 0}\phi\k({\frac{\k<{2^kx}>}{2^j}})a_{jkk},\\
R:\l_{p(k)}^{s}l_{\infty(j)}^{-\frac{1}{2}}L_{t,x}^2\l_{1(l)}^{\frac{1}{2}}\rightarrow X_p^s\qquad R:\l_{p(k)}^{s}L_t^\infty L_x^2\l_{1(j)}^{0}\l_{1(l)}^{0}\rightarrow W_p^s,
\end{gather*}
then we obtain $RQf=f$ for $f\in W_p^s\cap X_p^s$. 
Set $A=l_{\infty}^{-\frac{1}{2}}L_{t,x}^2\l_{1}^{\frac{1}{2}}\cap L_t^\infty L_x^2\l_{1}^{0}\l_{1}^{0}$, 
it is obvious that $W_p^s\cap X_p^s$ norm is equivalent to $\l_p^s(A)$ norm. 
By \cite[Theorem 5.6.3]{MR0482275} and \cite[Theorem 5.6.1]{MR0482275}, we have
$$
[\l_{2}^{s_0}(A),\l_{2}^{s_1}(A)]_\theta=\l_{2}^s (A),\ 
(\l_{p_0}^{s_0}(A),\l_{p_1}^{s_1}(A))_{\theta,p}=\l_{p}^s (A),\qquad 1\leq p_0,p_1,p\leq\infty,
$$
provided that $s_0\neq s_1$, $0<\theta<1$ and $s=(1-\theta)s_0+\theta s_1$. This completes the proof, in view of \cite[Theorem 6.4.2]{MR0482275}.

\subsection{Proof of \eqref{E1.5}}~

When $q=2$, we have $o=2$, $s=0$, and then
\cref{E1.5} follows directly from \cref{E3.3}. 
It remains to consider $2<q<\infty$, for which we
set $\alpha_0={q\alpha}/{2}$, $o_1={ o( q-2)}/{ (q- o)}$, $s_1={ q s}/{ (q-2)}$. 
Similarly 
to the proof in subsection \ref{sec-3.1},
we have $\alpha_0\leq 1/2$, 
$2\leq o_1\leq \infty$ and $1/2-1/o_1\leq s_1<n/2$, by 
 the conditions on $q$, $o$, $\alpha$ and $s$. 
Moreover, 
for $\theta=1-2/q\in (0,1)$, we have
\begin{Eq*}
\frac{1-\theta}{2}+\frac{\theta}{\infty}=\frac{1}{q},\qquad \frac{1-\theta}{2}+\frac{\theta}{o_1}=\frac{1}{o},\qquad (1-\theta)0+\theta s_1=s.
\end{Eq*}

Firstly when $q=o$ and $s=1/2-1/o$, 
where $o_1=\infty$ and $s_1=1/2-1/o_1$, 
by \cref{E3.3} and \cref{E3.4} we obtain 
\begin{Eq*}
\begin{cases}
A_{\alpha_0}(T)\|r^{-\frac{1}{2}+\alpha_0}\ppsi_1 u\|_{L_t^2\L_r^2 L_\omega^2}\lesssim\|u\|_{W^0\cap X^0},\\
\|r^{\frac{n-1}{2}}\ppsi_1 u\|_{L_t^\infty\L_r^{\infty}L_\omega^2}\lesssim\|u\|_{L_t^\infty\dot B_{2,1}^{\frac{1}{2}}}\lesssim\|u\|_{W_1^{\frac{1}{2}}\cap X_1^{\frac{1}{2}}}.\\
\end{cases}
\end{Eq*}
Recall the following facts
\begin{itemize}
\item \cite[Theorem 3.7.1]{MR0482275}: $(A_0,A_1)_{\theta,q}'=(A_0',A_1')_{\theta,q'}$ for $1\leq q<\infty$, $A_0\cap A_1$ is dense in $A_0$ and $A_1$;
\item \cite[1.18.4 eq.(3)]{MR503903}: $(L^{p_0}(A_0),L^{p_1}(A_1))_{\theta,p}=L^p((A_0,A_1)_{\theta,p})$ for $1\leq p_0,p_1<\infty$,  and $1/p=(1-\theta)/p_0+\theta/p_1$;
\item \cite[Theorem 5.5.1]{MR0482275}: $(L_{(\rho_0\d x)}^{p_0}(A),L_{(\rho_1\d x)}^{p_1}(A))_{\theta,p}=L_{(\rho\d x)}^{p}(A)$ for $1\leq p_0,p_1<\infty$, $1/p=(1-\theta)/p_0+\theta/p_1$ and $\rho=\rho_0^{p(1-\theta)/p_0}\rho_1^{p\theta/p_1}$,
\end{itemize}
we could use real interpolation with $(1-{2}/{q},q)$ for LHS to obtain
$$\textrm{LHS of }\cref{E1.5}\lesssim 
\|u\|_{(W^0\cap X^0, W_1^{\frac{1}{2}}\cap X_1^{\frac{1}{2}})_{\theta, q}}\ .$$
Applying \cref{E3.6} for RHS,
 we conclude
\begin{Eq*}
\textrm{LHS of }\cref{E1.5}\lesssim\|u\|_{W_q^{\frac{1}{2}-\frac{1}{q}}\cap X_q^{\frac{1}{2}-\frac{1}{q}}}\lesssim\|u\|_{W_2^{\frac{1}{2}-\frac{1}{q}}\cap X_2^{\frac{1}{2}-\frac{1}{q}}}= \textrm{RHS of }\cref{E1.5}, 
\end{Eq*}
where we have used the fact $l^2\subset l^q$ in the last inequality.

For the remaining case when $q>o$ or $s>1/2-1/o$,
we have $o_1<\infty$ or $1/2-1/o_1<s_1$. By \cref{E3.3} and \cref{E3.4}, we obtain
\begin{Eq*}
\begin{cases}
A_{\alpha_0}(T)\|r^{-\frac{1}{2}+\alpha_0}\ppsi_1 u\|_{L_t^2\L_r^2 L_\omega^2}\lesssim\|u\|_{W^0\cap X^0},\\
\|r^{\frac{n}{2}-\frac{n}{o_1}-s_1}\ppsi_1 u\|_{L_t^\infty\L_r^{o_1}L_\omega^2}\lesssim\|u\|_{W^{s_1}\cap X^{s_1}}.\\
\end{cases}
\end{Eq*}
To apply complex interpolation with $\theta=1-{2}/{q}$, 
 we use \cref{E3.5}  for RHS, and 
 record the following facts
 for LHS:
\begin{itemize}
\item \cite[Corollary 4.5.2]{MR0482275}: $[A_0,A_1]_{\theta}'=[A_0',A_1']_{\theta}$ for $A_0\cap A_1$ is dense in $A_0$ and $A_1$ and at least one of the spaces $A_0$ and $A_1$ is reflexive;
\item \cite[Theorem 5.1.2]{MR0482275}: $[L^{p_0}(A_0),L^{p_1}(A_1)]_{\theta}=L^p([A_0,A_1]_{\theta})$ for $1\leq p_0,p_1<\infty$,  and $1/p=(1-\theta)/p_0+\theta/p_1$;
\item \cite[Theorem 5.5.3]{MR0482275}: $[L_{(\rho_0\d x)}^{p_0}(A),L_{(\rho_1\d x)}^{p_1}(A)]_{\theta}=L_{(\rho\d x)}^{p}(A)$ for $1\leq p_0,p_1<\infty$, $1/p=(1-\theta)/p_0+\theta/p_1$ and $\rho=\rho_0^{p(1-\theta)/p_0}\rho_1^{p\theta/p_1}$.
\end{itemize}
This gives us \cref{E1.5} and completes the proof.

\subsection{Proof of \Le{L3.4}}~

For given $n\ge 2$, $o\in[2,\infty]$, 
$\alpha\in(1/2-1/o,1)$ and $\beta\in(\alpha-n/2,n/2)$,
we set 
$$\alpha_0=\frac{\alpha o + 2}{o + 2}\in(\frac{1}{2},1)\ ,\ 
\alpha_1=\frac{2\alpha o-o+2}{o + 2}\in(0,1)\ ,$$ 
$$\beta_0=\frac n2 -\frac{n-\alpha_0 }{n-\alpha }(\frac n2-\beta )\in(-\frac{n}{2}+\alpha_0,\frac{n}{2})\ ,$$ $$\beta_1=\frac n 2 -\frac{n-\alpha_1}{n-\alpha}(\frac n2-\beta )
\in(-\frac{n}{2}+\alpha_1,\frac{n}{2})
\ .$$
Let $\theta=2/o$, we have
$(1-\theta)\alpha_0+\theta\alpha_1=\alpha$, $(1-\theta)\beta_0+\theta\beta_1=\beta$.
Recall that by \cite[(4.10)]{MR3691393}, we have
\begin{Eq}\label{E3.10}
\|r^{\beta} D^{\alpha} u\|_{L_x^2}\approx \|r^{\beta} 2^{\alpha k} S_k u\|_{l_{k\in\Z}^2L_x^2}\approx \| 2^{\beta j+\alpha k} \phi_jS_k u\|_{l_{j\in\Z}^2l_{k\in\Z}^2L_x^2},
\end{Eq}
for any $|\beta|<n/2$.
Then by weighted trace estimate of \cite[Lemma 4.2]{MR3691393} and  weighted \emph{Hardy-Littlewood-Sobolev} estimates of {Stein-Weiss}, 
we obtain
\begin{Eq*}
\begin{cases}
\|r^{\frac{n}{2}-\alpha_0+\beta_0}u\|_{\L_r^\infty L_\omega^2}
\lesssim \|r^{\beta_{0}}D^{\al_{0}}u\|_{L^{2}}
\les  \| 2^{\beta_{0} j+\alpha_{0} k} \phi_jS_k u\|_{l_{j\in\Z}^2l_{k\in\Z}^2L_x^2}
\ ,\\
\|r^{-\alpha_1+\beta_1}u\|_{\L_r^2L_\omega^2}\lesssim\|r^{\be_{1}}D^{\al_{1}}   u\|_{L^2}
\les  \| 2^{\beta_{1} j+\alpha_{1} k} \phi_jS_k u\|_{l_{j\in\Z}^2l_{k\in\Z}^2L_x^2}
\ ,
\end{cases}
\end{Eq*}
where
$D=\sqrt{-\Delta}$.
Here we should mention that
the first inequality was stated for $n\ge 3$ and $\al_{0}\in (1/2, 1]$ in
\cite[Lemma 4.2]{MR3691393}. However, we observe that the same proof
also apply in the situation of
$n=2$ and  $\alpha_0\in (1/2, 1)$.

We remark that we can do complex interpolation for RHS,
if we follow a similar proof as in subsection \ref{sec-3.2}.
For LHS, we record the following facts:
\begin{itemize}
\item \cite[Corollary 4.5.2]{MR0482275}: $[A_0,A_1]_{\theta}'=[A_0',A_1']_{\theta}$ for $A_0\cap A_1$ is dense in $A_0$ and $A_1$ and at least one of the spaces $A_0$ and $A_1$ is reflexive;
\item \cite[Theorem 5.5.3]{MR0482275}: $[L_{(\rho_0\d x)}^{p_0}(A),L_{(\rho_1\d x)}^{p_1}(A)]_{\theta}=L_{(\rho\d x)}^{p}(A)$ for $1\leq p_0,p_1<\infty$, $1/p=(1-\theta)/p_0+\theta/p_1$ and $\rho=\rho_0^{p(1-\theta)/p_0}\rho_1^{p\theta/p_1}$.
\end{itemize}
Thus, by complex interpolation with $\theta={2}/{o}$, together with
 \cref{E3.10},
 we get
\begin{Eq*}
\|r^{\frac{n}{2}-\frac{n}{o}-\alpha+\beta}u\|_{\L_r^o  L_\omega^2}\lesssim\| 2^{\beta j+\alpha k} \phi_jS_k u\|_{l_{j\in\Z}^2l_{k\in\Z}^2L_x^2}\les
\| r^{\be}D^{\alpha } u\|_{L_x^2}\ .
\end{Eq*}
This completes the proof.

\subsection{Proof of \Le{L3.5}}~

We follow  the similar proof of \cite[Lemma 5.5]{MR3691393}. Take $\delta>0$ and small enough such that $s>1/2-1/o+\de$, 
then we conclude
\begin{Eq*}
\|r^{\frac{n-1}{2}-\frac{n}{o}-s}u\|_{L_t^2\L_r^oL_\omega^2}\lesssim\|r^{-s-\frac{1}{o}+\delta}D^{\frac{1}{2}-\frac{1}{o}+\delta}u\|_{L_t^2L_x^2}\lesssim\|D^{\frac{1}{2}-\frac{1}{o}+\delta}u\|_{X^{s-\frac{1}{2}+\frac{1}{o}-\delta}}\lesssim\|u\|_{X^s},
\end{Eq*}
where we have used \cref{E3.7} in the first inequality and \cref{E3.2} in the second inequality.

\subsection{Proof of \eqref{E1.6}}~

The proof of \cref{E1.6} is very similar to the proof of \cref{E1.5}. 
Let $o_0=o_1=o$ and $\theta=2/q$, 
we set $$s_0=s-\frac{2os - o + 2}{2o - 2q + oq - noq}
\in (\frac{1}{2}-\frac{1}{o_0}, \frac{n}{2})\ ,$$
$$s_1=s+\frac{2os - o + 2}{4o - 4q + 2oq - 2noq}(q - 2)
\in (\frac{1}{2}-\frac{1}{o_1} , \frac{n-1}{2})\ ,
$$ such that
$s_{0}(1-\theta)+s_{1}\theta=s$.
By \cref{E3.4} and \cref{E3.8} we obtain
\begin{Eq*}
\begin{cases}
\|r^{\frac{n}{2}-\frac{n}{o_0}-s_0}u\|_{L_t^\infty\L_r^{o_0}L_\omega^2}\lesssim\|u\|_{W^{s_0}\cap X^{s_0}},\\
\|r^{\frac{n-1}{2}-\frac{n}{o_1}-{s_1}}u\|_{L_t^2\L_r^{o_1}L_\omega^2}\lesssim\|u\|_{W^{s_1}\cap X^{s_1}}.
\end{cases}
\end{Eq*}
Using complex interpolation with $\theta=2/q$, 
we get \cref{E1.6}.

%
%
%
%

\section{Restatement and proof of \Th{T1.2}}\label{S4}

In this section, 
we will establish  our main result for the nonlinear problem,
 \Th{T1.2}.

\subsection{Restatement and a Corollary}~

Firstly we give the restatement of \Th{T1.2}.

\begin{thm}\label{T4.1}
Let $n=3$, $p_1, p_2\ge 2$,  and assume \eqref{H1} and \eqref{H2}. 
Fix $\delta>0$, 
set
\begin{Eq}\label{E4.1}
s_\iota\dyw\frac{3}{2}+\sigma-\frac{2(p_\iota + 1)}{p_1p_2 - 1},
\end{Eq}
where $\sigma$ is defined in \cref{E1.3}. 
Then, there exists $R,c,\varepsilon_0>0$ such that for any $(f_\iota,g_\iota)$ satisfying
\begin{Eq}\label{E4.2}
\sum_\iota \|Y^{\leq 2}(f_\iota,g_\iota)\|_{H_x^1\times L_x^2}+\|\ppsi_R Y^{\leq 2}g_\iota\|_{\H_x^{s_\iota-1}} =\varepsilon\leq\varepsilon_0\ ,
\end{Eq}
with additional condition $\supp(f_\iota,g_\iota)\subset B_R$ when $p_{\hiota}=2$, 
there is a unique solution $u_\iota$ of \cref{E1.2} in $M\cap([0,T_\varepsilon)\times\R^3)$ where $T_\varepsilon$ is defined in \cref{E1.7}, 
such that $\|\ppsi_R Z^{\leq 2}u_\iota \|_{W^s\cap X^s}+\|Z^{\leq 2}u_\iota\|_{L\!E_\g}\lesssim\varepsilon$.  
At last, 
if \eqref{H3} is satisfied, 
 the initial data do not need to be compactly supported.
\end{thm}

When 
$F_{p_1}=F_{p_2}$ and $(f_{1}, g_{1})=(f_{2}, g_{2})$,
we can obtain the following corollary, for $\square_\g u=F_p(u)$.

\begin{coro}\label{C4.2}
Let $n=3$, 
$p\ge 2$ and assume \eqref{H1} and \eqref{H2}. 
Set
\begin{Eq*}
s\dyw\frac{1}{2}-\frac{1}{p},
\end{Eq*}
then, 
there exists $R,c,\varepsilon_0>0$ such that for any $(f,g)$ satisfying
\begin{Eq*}
\|Y^{\leq 2}(f,g)\|_{H_x^1\times L_x^2}+\|\ppsi_R Y^{\leq 2}g\|_{\H_x^{s-1}} =\varepsilon\leq\varepsilon_0
\end{Eq*}
and $\supp(f,g)\subset B_R$ if $p=2$, 
there is a unique solution $u$ of
\begin{Eq*}
\begin{cases}
\square_\g u=F_p(u),~(t,x)\in M\ ,\\
(u,\partial_t u)|_{t=0}=(f,g)\ ,\\
\end{cases}
\end{Eq*}
in $M\cap([0,T_\varepsilon')\times\R^3)$, 
such that $u$ satisfies $\|\ppsi_R Z^{\leq 2}u \|_{W^s\cap X^s}+\|Z^{\leq 2}u\|_{L\!E_\g}\lesssim\varepsilon$. 
Here $T_\varepsilon'$ is defined by
\begin{Eq*}T_\varepsilon'\dyw 
\begin{cases}
c\varepsilon^{-p(p-1)/(1+2p-p^2)}&2\leq p<1+\sqrt2,\\
\exp(c\varepsilon^{-2\sqrt2})&p=1+\sqrt2,\\
\infty&p>1+\sqrt2\ .
\end{cases}
\end{Eq*}
\end{coro}

\subsection{Preparation}~

In order to facilitate our later proof, 
we give some lemmas which will be used in this section.

\begin{prop}[Lemma 4.3 in \cite{MR3691393}]
Let $n\geq 2$, 
$R\geq 3$ and $k=\lfloor (n+2)/2\rfloor$ be the integral part of $(n+2)/2$. 
We have
\begin{Eq}\label{E4.3}
\|r^{b}u\|_{L^q_{r} L^\infty_\omega(r\ge R+1)}\lesssim  \|r^{b-(n-1)/p+(n-1)/q} Y^{\leq k} u\|_{L^p_{r}L^{2}_\omega(r\geq R)}
\end{Eq}
for any $b\in\R$, 
$2\leq p\leq q\leq \infty$. 
Moreover, 
for any $2\leq p\leq q\leq {q_2}<\infty$, 
$b\in\R$, $s\in \Z$ with
$s\geq n/2-n/{q_2}$, 
we have
\begin{Eq}\label{E4.4}
\|r^{b}u\|_{L^q_{r} L^{q_2}_\omega(r\geq R+1)}\lesssim  \|r^{b-(n-1)/p+(n-1)/q} Y^{\leq s} u\|_{L^p_{r}L^{2}_\omega(r\geq R)}.
\end{Eq}
\end{prop}

\begin{prop}[Theorem 5.9 in \cite{MR3691393}] \label{P4.3}
Let $n\geq3$, 
and assume \eqref{H1} and \eqref{H2}. 
Then we have
\begin{Eq*}
\|Z^{\leq k}u\|_{L\!E_\g}\lesssim&\|Y^{\leq k}u(0,x)\|_{H_x^1}+\|Y^{\leq k}\partial_tu(0,x)\|_{L_x^2}+\|Z^{\leq k}F\|_{L_t^1L_x^2},
\end{Eq*}
for any $u$ satisfies $\square_\g u=F$ and any $k\geq 0$. 
\end{prop}

\begin{lemma}\label{L4.5}
For $n=3$, 
consider
\begin{Eq*}
P=\partial_t^2-\Delta+\ga^{\mu\nu}(t,x)\partial_\mu\partial_\nu+\b^\mu(t,x)\partial_\mu,
\end{Eq*}
with $\ga^{\mu\nu}=\ga^{\nu\mu}\in C^2$, 
$\b^\mu\in C^1$, 
and
\begin{Eq*}
\|\ga\|_{\l_1^0L_{t,x}^\infty}+\|(\partial \ga,\b)\|_{\l_1^1L_{t,x}^\infty}+\|(\partial^2\ga,\partial\b)\|_{\l_1^2L_{t,x}^\infty}\leq \delta\ .
\end{Eq*}
Then
there exists $\delta>0$ so that
 we have
\begin{Eq}\label{E4.5}
\|\partial u\|_{W^s\cap X^s}\lesssim\|\partial u(0)\|_{\H_x^s}+\|Pu\|_{(W^{-s}\cap X^{-s})'},
\end{Eq}
for any $|s|<1$.
\end{lemma}

\begin{proof}
This lemma is a simple modification of \cite[Theroem 1]{MR2944027}. For the convenience of the reader, we adopt the notation in \cite{MR2944027} and give a sketch of the proof. 
By \cite[eq.(50)]{MR2944027} we have
\begin{Eq*}
\|\partial u\|_{L_t^\infty L_x^2\cap X_k}\lesssim\|\partial u(0)\|_{L_x^2}+\|P_{(k)}u\|_{L_t^1L_x^2+X_k'}.
\end{Eq*}
Then after summation we obtain
\begin{Eq*}
\|\partial u\|_{W^s\cap X^s}^2\lesssim& \sum_k 2^{2sk}\|\partial S_ku\|_{L_t^\infty L_x^2\cap X_k}^2\\
\lesssim&\sum_k 2^{2sk}\k({\|S_k\partial u(0)\|_{L_x^2}^2+\|P_{(k)}S_ku\|_{L_t^1 L_x^2+ X_k'}^2})\\
\lesssim&\|\partial u(0)\|_{H_x^s}^2+\|Pu\|_{(W^{-s}\cap X^{-s})'}^2+\delta\|\partial u\|_{X^s}^2,
\end{Eq*}
which gives us
\cref{E4.5} provided that $\delta$ is small. 
\end{proof}

\begin{coro}\label{C4.6}
Let $n=3$,   $s>0$ and $q, o, \alpha, s$ satisfy the condition of \cref{E1.5}.
Assume \eqref{H1},
then there exists $K\geq R_0$ such that
\begin{Eq}\label{E4.6}
\|u\|_{\DS_{q,o,\alpha,s}\cap L_t^\infty \H_x^s}\lesssim&\|f\|_{\H_x^{s}}+\|g\|_{\H_x^{s-1}}+\|F\|_{(W^{1-s}\cap X^{1-s})'}\ ,
\end{Eq}
for any $u$ satisfying $\square_\g u=F$ with initial data $(f,g)$ on $\{t=0\}$ and vanishing in the region $\{r<K\}$. 
\end{coro}

\begin{proof}
Based on \eqref{H1}, 
we set $\tilde\g=\psi_{K-1}\m+\ppsi_{K-1}\g$ with $K\geq R_0+1$ such that $\tilde\g$ satisfies the assumption in \Le{L4.5}. 
Since we are assuming that $u$ vanishes in the region $\{r<K\}$, 
we have $\square_{\g}u=\square_{\tilde\g}u$. 
Then, 
by \cref{E1.5} and \cref{E4.5}, we get
$$
\|u\|_{\DS_{q,o,\alpha,s}\cap L_t^\infty \H_x^s}\les 
\|u\|_{W^s\cap X^s}\les
\|\nabla u\|_{W^{s-1}\cap X^{s-1}}
\les
\|f\|_{\H_x^{s}}+\|g\|_{\H_x^{s-1}}+\|F\|_{(W^{1-s}\cap X^{1-s})'}\ ,
$$
where we have assumed $s>0$ so that $|s-1|<1$ to apply \cref{E4.5}.
\end{proof}

\subsection{Proof of \Th{T4.1} with $p_1>2$}~

Now we begin the proof of \Th{T4.1}.  To close the estimates we define $\DS_\iota=\DS_{q_\iota,o_\iota,\alpha_\iota,s_\iota}$ and $\widetilde\DS_\iota=\DS_{\tilde q_\iota,\tilde o_\iota,\tilde \alpha_\iota,s_\iota}$ with $s_\iota$ defined in \cref{E4.1} and
\begin{Eq}\label{E4.7}
q_\iota\dyw& 2p_\hiota,&~ o_\iota\dyw&2p_\hiota-p_1,&~ \alpha_\iota\dyw&\frac{\sigma(p_\hiota-1)}{p_\hiota},\\
\tilde q_\iota\dyw& p_\hiota,&~ \tilde o_\iota\dyw& p_\hiota,&~ \tilde\alpha_\iota\dyw& s_\iota-\frac{3}{2}+\frac{3}{p_\hiota}.
\end{Eq}
Then we define 
\begin{alignat}{1}
\label{E4.8}\|u\|_{U_\iota^m}=&\|\ppsi_R Z^{\leq m}u\|_{\DS_{\iota}\cap \widetilde{\DS}_{\iota}}+\|Z^{\leq m}u\|_{L\!E_{\g}},\\
\label{E4.9}\|F\|_{V_\iota^m}=&\|\ppsi_R^{p_\iota} r^{3-{s_\iota}-\frac{3p_\iota}{o_\hiota}} Z^{\leq m}F\|_{L_t^2\L_r^{\frac{o_\hiota}{p_\iota}}L_\omega^2}+\|Z^{\leq m}F\|_{L_t^1L_x^2}.
\end{alignat}
Here we mention that 
$s_{\iota}>0$ as $p_{1}>2$,
the indices defined in \cref{E4.7} satisfy 
$\tilde\alpha_\iota<0$ and
the condition of \cref{E1.5}.
Moreover, the indices $q=2$, 
$o=(o_\hiota/p_\iota)'$, 
$s=1-s_\iota$ satisfy the condition of \cref{E1.6}, 
and
\begin{Eq}\label{E4.10}
3-s_\iota-\frac{3p_\iota}{o_\hiota}=p_\iota(\frac{3}{2}+\alpha_\hiota-\frac{1}{q_\hiota}-\frac{3}{o_\hiota}-s_\hiota).
\end{Eq}

Next, 
we present the linear estimates, which will be proved in subsection \ref{sec:4.4}.
\begin{lemma}\label{L4.7}
For $2< p_1\leq p_2$, assume \eqref{H1} and \eqref{H2}, 
there exists $R>0$ big enough such that for any $m\geq 0$
\begin{Eq}\label{E4.11}
\|u_\iota\|_{U_\iota^m}\lesssim& \|Y^{\leq m}f_\iota\|_{H_x^1}+\|Y^{\leq m}g_\iota\|_{L_x^2}+\|\ppsi_R Y^{\leq m}g_\iota\|_{\H_x^{s_\iota-1}}\\
&+\|\ppsi_RZ^{\leq m-1}F_{\iota}(0,x)\|_{\H_x^{s_\iota-1}}+\|F_\iota\|_{V_\iota^m},
\end{Eq}
where $u_\iota$ is the solution of
\begin{Eq*}
\square_\g u_\iota=F_\iota,\qquad u_\iota(0,x)=f_\iota(x),\qquad \partial_tu_\iota(0,x)=g_\iota(x).
\end{Eq*}
\end{lemma}

On the other hand, 
we have the following  nonlinear estimates, to be proved in subsection \ref{sec:4.5}.
\begin{lemma}\label{L4.8}
For $2< p_1\leq p_2$, 
we have
\begin{alignat}{1}
\label{E4.12}\|F_\iota(u_\hiota)\|_{V_\iota^2}\lesssim& A_{q_\hiota\alpha_\hiota/2}(T)^{-\frac{2p_\iota}{q_\hiota}}\|u_\hiota\|_{U_\hiota^2}^{p_\iota},\\
\label{E4.13}\|F_\iota(u_\hiota)-F_\iota(v_\hiota)\|_{V_\iota^0}\lesssim&  A_{q_\hiota\alpha_\hiota/2}(T)^{-\frac{2p_\iota}{q_\hiota}}\|(u_\hiota,v_\hiota)\|_{U_\hiota^2}^{p_\iota-1}\|u_\hiota-v_\hiota\|_{U_\hiota^0}.
\end{alignat}
\end{lemma}

Equipped with \Le{L4.7} and \Le{L4.8}, 
it is a standard procedure to prove \Th{T4.1} with $p_{1}, p_{2}>2$. Let $T_\varepsilon$ be given in \cref{E1.7} and $c, \ep$ small enough such that $A_{q_\hiota\alpha_\hiota/2}(T_\varepsilon)^{-{2p_\iota}/{q_\hiota}}\varepsilon^{p_\iota-1}\ll 1$,
we could iterate in $X_k\dyw\{(u_1,u_2):\max_\iota\|u_\iota\|_{U_\iota^k}\leq\varepsilon\}$, 
with the time interval $[0,T_\varepsilon]$.
By \cref{E4.11}, 
\cref{E4.12} and \cref{E4.13}, we know that the sequence is well defined in $X_2$ and contractive in $X_0$. 
Then we obtain the fixed point $(u_1,u_2)\in X_2$ solving \cref{E1.2}.
Again, by \cref{E4.11} and \cref{E4.13}, we know the solution is unique, 
which finishes the proof.

\subsection{Proof of \Le{L4.7}}\label{sec:4.4}~

Without loss of generality we omit the index $\iota$. 
For the $\DS$ and $\widetilde \DS$ norm defined in \cref{E4.8}, 
set $R\geq K$ big enough such that $w\dyw\ppsi_R Z^{\leq m}u$ satisfies \Co{C4.6}, 
we obtain
\begin{Eq*}
\|w\|_{\DS\cap\widetilde\DS}\lesssim\|\partial w(0,x)\|_{\H_x^{s-1}}+\|\square_{\g}w\|_{(W^{1-s}\cap X^{1-s})'}.
\end{Eq*}

Similar to the proof of \cite[Lemma 5.6.]{MR3691393} we have
\begin{Eq}\label{E4.14}
\square_{\g}w&=[\square_\g,\ppsi_R]Z^{\leq m}u+\ppsi_R[\square_\g,Z^{\leq m}]u+\ppsi_R Z^{\leq m}F.
\end{Eq}
Now, 
we can control each term in $\square_{\g}w$ separately. 
For the first term in \cref{E4.14}, 
with the help of the duality of \cref{E3.2}, 
we see
\begin{Eq*}
\|[\square_\g,\ppsi_R]Z^{\leq m}u\|_{(W^{1-s}\cap X^{1-s})'}\lesssim&\|\chi_{r\in [R,R+1]} \mathcal{O}(\partial^{\leq {m+1}} u)\|_{(W^{1-s}\cap X^{1-s})'}\\
\lesssim&\|r^{\frac{3}{2}-s}\chi_{r\in [R,R+1]} \mathcal{O}(\partial^{\leq {m+1}} u)\|_{L_t^2L_x^2}\\
\lesssim&\|\partial^{\leq m+1}u\|_{L_t^2L_x^2(r\in[R,R+1])}.
\end{Eq*}
For the second term in \cref{E4.14}, 
we have \cite[eq.(5.25)]{MR3691393}
\begin{Eq}\label{E4.15}
\ppsi_R[\square_\g,Z^{\leq m}]u=\partial_x(\beta_1\ppsi_R \partial Z^{\leq {m-1}}u)+\beta_2 \ppsi_R \partial Z^{\le {m-1}}u+
\beta_1 \ppsi_R Z^{\le {m-1}}F,
\end{Eq}
with \cite[eq.(5.24)]{MR3691393}
\begin{Eq*}
\|\beta_1\|_{l_1^{1} L_t^\infty L_x^\infty}+\||\partial \beta_1|+|\beta_2|\|_{l_1^{2} L_t^\infty L_x^\infty}\lesssim 1 .
\end{Eq*}
Since we always have $0< s<1$, 
with the help of the duality of \cref{E3.1}, 
the first term in \cref{E4.15} with $(W^{1-s}\cap X^{1-s})'$ norm can be controlled by
\begin{Eq*}
& \|\partial_x(\beta_1\ppsi_R \partial Z^{\leq {m-1}}u)\|_{(X^0)'\cap (X^1)'}\\
\lesssim&\|\partial_x(\beta_1\ppsi_R \partial Z^{\leq {m-1}}u)\|_{l_1^{\frac{1}{2}}L_t^2L_x^2}+\|\beta_1\ppsi_R \partial Z^{\leq {m-1}}u\|_{l_1^{\frac{1}{2}}L_t^2L_x^2}\\
\lesssim&\|\ppsi_R \partial Z^{\leq {m}}u\|_{l_\infty^{-\frac{1}{2}}L_t^2L_x^2}+\|\partial^{\leq m}u\|_{L_t^2L_x^2(r\in[R,R+1])}.
\end{Eq*}
With the help of the duality of \cref{E3.2}, 
the second term in \cref{E4.15} with $(W^{1-s}\cap X^{1-s})'$ norm can be controlled by
\begin{Eq*}
\|\beta_2 \ppsi_R \partial Z^{\le {m-1}}u\|_{l_2^{\frac{3}{2}-s}L_t^2L_x^2}\lesssim \|\ppsi_R \partial Z^{\le {m-1}}u\|_{l_\infty^{-\frac{1}{2}}L_t^2L_x^2}.
\end{Eq*}
By duality of \cref{E1.6} and \emph{Minkowski} inequality, 
the third term in \cref{E4.15} together with the last term in \cref{E4.14} with $(W^{1-s}\cap X^{1-s})'$ norm can be controlled by
\begin{Eq*}
&\|(\ppsi_R-\ppsi_R^p)(\beta_1 Z^{\leq {m-1}}F+Z^{\leq m}F)\|_{l_2^{1-s}L_t^1L_x^2}\\
&+\|r^{3-s-\frac{3p}{o}}\ppsi_R^p(\beta_1 Z^{\leq {m-1}}F+Z^{\leq m}F)\|_{L_t^2\L_r^{\frac{o}{p}}L_\omega^2}\\
\lesssim&\|\partial^{\leq m}F\|_{L_t^1L_x^2(r\in[R,R+1])}+\|r^{3-s-\frac{3p}{o}}\ppsi_R^pZ^{\leq m}F\|_{L_t^2\L_r^{\frac{o}{p}}L_\omega^2}.
\end{Eq*}

In summary,
recall the definition of $\|\cdot\|_{LE_\g}$ in \cref{E1.4}, 
we conclude
\begin{Eq*}
\|u\|_{U^m}\lesssim& \|\partial^{\leq 1}Z^{\leq m}u(0,x)\|_{\H_x^{s-1}}+\|Z^{\leq m}u\|_{L\!E_\g}+\|F\|_{V^m}.
\end{Eq*}
Then we get \Le{L4.7} from \Pr{P4.3} and a direct calculation as \cite{MR3691393}.

\subsection{Proof of \Le{L4.8}}\label{sec:4.5}~

Here, 
we only give the proof of \cref{E4.12} and omit the similar proof of \cref{E4.13}. 
Noticing that
\begin{Eq*}
|Z^{\leq 2}F_\iota(u_\hiota)|\lesssim |u_\hiota|^{p_\iota-1}|Z^{\leq2}u_\hiota|+|u_\hiota|^{p_\iota-2}|Z u_\hiota|^2\equiv \Roma1_\iota+\Roma2_\iota,
\end{Eq*}
which will be handled separately. 

To control the first semi-norm in \cref{E4.9},
we apply \emph{Sobolev}'s inequality on sphere to get
\begin{eqnarray*}
 \|Z^{\leq 2}F_\iota(u_\hiota)\|_{L_\omega^2} 
 & \les & 
\| u_\hiota\|_{L_\omega^\infty} ^{p_\iota-1}
\|Z^{\leq2}u_\hiota\|_{L_\omega^2}
+\| u_\hiota\|_{L_\omega^\infty} ^{p_\iota-2}
\|Z^{\leq 1}u_\hiota\|_{L_\omega^4}^{2}
 \\
 & \les & 
\|Z^{\leq2}u_\hiota\|_{L_\omega^2}^{p_\iota}\ .
\end{eqnarray*}
By \cref{E4.10}, 
\cref{E1.5}
and
\emph{H\"older's} inequality,
we see that
\begin{Eq*}
\|\ppsi_R^{p_\iota} r^{3-{s_\iota}-\frac{3p_\iota}{o_\hiota}} Z^{\leq 2}F_\iota(u_\hiota)\|_{L_t^2\L_r^{\frac{o_\hiota}{p_\iota}}L_\omega^2}\lesssim&\|\ppsi_R r^{\frac{3-{s_\iota}-\frac{3p_\iota}{o_\hiota}}{p_\iota}} Z^{\leq 2}u_\hiota\|_{L_t^{2p_\iota}\L_r^{o_\hiota}L_\omega^{2}}^{p_\iota}\\
=&\|\ppsi_R r^{\frac{3}{2}+\alpha_\hiota-\frac{1}{q_\hiota}-\frac{3}{o_\hiota}-s_\hiota} Z^{\leq 2}u_\hiota\|_{L_t^{q_\hiota}\L_r^{o_\hiota}L_\omega^2}^{p_\iota}\\
\lesssim&  A_{q_\hiota\alpha_\hiota/2}(T)^{-\frac{2p_\iota}{q_\hiota}}\|u_\hiota\|_{U_\hiota^2}^{p_\iota}.
\end{Eq*}

To control the second semi-norm in \cref{E4.9}, we begin by dealing with $\Roma1_\iota$. For the part $r\ge R+2$, we observe that
\begin{Eq*}
\|\Roma1_\iota\|_{L_t^1\L_r^2L_\omega^2(r\geq R+2)}\lesssim&\|r^{\frac{1}{p_\iota(p_\iota-1)}}u_\hiota\|_{L_t^{p_\iota}\L_r^{\frac{2p_\iota(p_\iota-1)}{(p_\iota-2)}}L_\omega^\infty(r\geq R+2)}^{p_\iota-1}\|r^{-\frac{1}{p_\iota}}\ppsi_RZ^{\leq 2}u_\hiota\|_{L_t^{p_\iota}\L_r^{p_\iota}L_\omega^2}\\
\lesssim& \|r^{-\frac{1}{p_\iota}}\ppsi_R Z^{\leq 2}u_\hiota\|_{L_t^{p_\iota}\L_r^{p_\iota}L_\omega^2}^{p_\iota}=\|\ppsi_R Z^{\leq 2} u_\hiota\|_{\widetilde{\DS}_{\hiota}}^{p_\iota}\lesssim\|u_\hiota\|_{U_\hiota^2}^{p_\iota},
\end{Eq*}
where we have applied \cref{E4.3} in the second inequality. 
Meanwhile, for the spatial compact region $r\le R+2$,
by \emph{Sobolev} embedding $\H_x^{2}\cap \H_x^1\subset L_x^\infty$, 
we see
\begin{Eq*}
\|\Roma1_\iota\|_{L_t^1 L_x^2(r\leq R+2)}\lesssim&\|u_\hiota\|_{(L_t^{2}\cap L_t^\infty) L_x^\infty(r\leq R+2)}^{p_\iota-1}\|\partial^{\leq 2}u_\hiota\|_{L_t^{2} L_x^2(r\leq R+2)}\\
\lesssim& \|\partial^{\leq 2}u_\hiota\|_{L\!E_\g}^{p_\iota}\lesssim\|u_\hiota\|_{U_\hiota^2}^{p_\iota}.
\end{Eq*}

The treatment of $\Roma2_\iota$ is similar.
For the spatial compact region,
we obtain 
\begin{Eq*}
\|\Roma2_\iota\|_{L_t^1L_x^2(r\leq R+2)}\lesssim&\|u_\hiota\|_{ L_t^\infty L_x^\infty(r\leq R+2)}^{p_\iota-2}\|\pa u_\hiota\|_{L_t^{2} L_x^4(r\leq R+2)}^2\\
\lesssim& \|\partial^{\leq 2}u_\hiota\|_{L\!E_\g}^{p_\iota}\lesssim\|u_\hiota\|_{U_\hiota^2}^{p_\iota}\ ,
\end{Eq*}
 by $\H_x^{2}\cap \H_x^1\subset L_x^\infty$ and $H_x^1\subset L_x^4$.
For the remaining region $r\ge R+2$, we deal with two cases separately: $p_\iota\leq 4$ and $p_\iota> 4$. At first, when $p_\iota\leq 4$,
with the help of \cref{E4.3} and \cref{E4.4},  we have
\begin{Eq*}
\|\Roma2_\iota\|_{L_t^1\L_r^2L_\omega^2(r\geq R+2)}\lesssim&\|r^{\frac{1}{p_\iota}}u_\hiota\|_{L_t^{p_\iota}\L_r^\infty L_\omega^\infty(r\geq R+2)}^{p_\iota-2}\|r^{\frac{1}{p_\iota}-\frac{1}{2}}Z u_\hiota\|_{L_t^{p_\iota}\L_r^4L_\omega^4(r\geq R+2)}^2\\
\lesssim& \|r^{-\frac{1}{p_\iota}}\ppsi_RZ^{\leq 2}u_\hiota\|_{L_t^{p_\iota}\L_r^{p_\iota}L_\omega^2}^{p_\iota}=\|\ppsi_R Z^{\leq 2} u_\hiota\|_{\widetilde{\DS}_{\hiota}}^{p_\iota}\lesssim\|u_\hiota\|_{U_\hiota^2}^{p_\iota}\ .
\end{Eq*}
When $p_\iota>4$, by \cref{E4.3}, we  get
\begin{Eq*}
\|\Roma2_\iota\|_{L_t^1\L_r^2L_\omega^2(r\geq R+2)}\lesssim&\|r^{\frac{2}{p_\iota(p_\iota-2)}}u_\hiota\|_{L_t^{p_\iota}\L_r^\frac{2p_\iota(p_\iota-2)}{p_\iota-4} L_\omega^\infty(r\geq R+2)}^{p_\iota-2}\|r^{-\frac{1}{p_\iota}}Z u_\hiota\|_{L_t^{p_\iota}\L_r^{p_\iota}L_\omega^{p_\iota}(r\geq R+2)}^2\\
\lesssim& \|r^{-\frac{1}{p_\iota}}\ppsi_R Z^{\leq 2}u_\hiota\|_{L_t^{p_\iota}\L_r^{p_\iota}L_\omega^2}^{p_\iota}=\|\ppsi_R Z^{\leq 2} u_\hiota\|_{\widetilde{\DS}_{\hiota}}^{p_\iota}\lesssim\|u_\hiota\|_{U_\hiota^2}^{p_\iota}.
\end{Eq*}
In summary,  we finish the proof of \cref{E4.12}, as we have
$A_{q_\hiota\alpha_\hiota/2}(T)\le 1$ by definition.

\subsection{Proof of \Th{T4.1} with $p_1=2$}~

Firstly we consider the case $p_2>p_1=2$.
The main difference here is that,
$s_2$ defined in \cref{E4.1} is zero but \Co{C4.6} requires $s_\iota>0$. 

To overcome the difficulty, 
we follow the idea of \cite{LMSTW}.
The key observation is that,
although the homogeneous estimate is hard to prove or even fail to hold,
we could still prove the inhomogeneous estimate in \cref{E4.6} with $s=0$.
Actually,
by \cref{E4.5} with $s=0$,
if $v$ has vanishing initial data at $t=T$,
we obtain, for $t\in [0, T]$,
\begin{Eq*}
\|v\|_{W^1\cap X^1}\lesssim \|\square_{\tilde\g} v\|_{(W^0\cap X^0)'},
\end{Eq*} 
where $\tilde\g$ is the same as that in the proof of \Co{C4.6}. 
Then,
by  duality for  the wave  operator $\square_{\tilde\g}$, we still have the desired inhomogeneous estimate
\begin{Eq*}
\|u\|_{W^0\cap X^0}\lesssim\|\square_{\tilde\g} u\|_{(W^1\cap X^1)'}
=\|\square_\g u\|_{(W^1\cap X^1)'}\ ,
\end{Eq*} 
provided that $u$ has vanishing initial data at $t=0$
and vanishes in a region $\{r<R+1\}$ with $R\gg 1$. 
Therefore, 
we require $\supp (f_2,g_2)\subset B_R$ to ensure that $\ppsi_R Z^{\leq 2}u_2$ has zero initial data.

Another difference in $p_1=2$ is that in this situation, $q=2$, $s=1$ no longer satisfy the condition of \cref{E1.6}, so we modify the definition of $q_\iota$ in \cref{E4.7} to $q_\iota\dyw (2-\delta)p_\hiota$ and correspondingly modify \cref{E4.9} to
\begin{Eq*}
\|F\|_{V_\iota^m}=&\|\ppsi_R^{p_\iota} r^{\frac{7}{2}-\frac{1}{2-\delta}-{s_\iota}-\frac{3p_\iota}{o_\hiota}} Z^{\leq m}F\|_{L_t^{2-\delta}\L_r^{\frac{o_\hiota}{p_\iota}}L_\omega^2}+\|Z^{\leq m}F\|_{L_t^1L_x^2}\ .
\end{Eq*}
Then $q=(2-\delta)'$, $s=1$ satisfy the condition of \cref{E1.6}, and we obtain
\begin{Eq*}
\frac{7}{2}-\frac{1}{2-\delta}-s_\iota-\frac{3p_\iota}{o_\hiota}=p_\iota(\frac{3}{2}+\alpha_\hiota-\frac{1}{q_\hiota}-\frac{3}{o_\hiota}-s_\hiota).
\end{Eq*}
The rest of the proof is almost the same, and we leave the details for the interested reader.

As for the second case $p_1=p_2=2$, we shall require $\supp(f_j, g_j)\subset B_R$, and then the proof is similar.

Finally, if we additionally assume \eqref{H3},  we could recover the homogeneous estimate in \cref{E4.6} with $s=0$, as in \cite[eq. (6.13)]{MR3691393}. With the help of this observation, we could prove the same result, without the compact requirement.

\subsection*{Acknowledgment}
This work was supported by NSFC 11671353 and 11971428.

\end{document}